\documentclass[11pt,reqno]{amsart}
\usepackage[utf8]{inputenc}
\usepackage[T1]{fontenc}
\usepackage[initials,nobysame,shortalphabetic]{amsrefs}
\usepackage{enumerate}
\usepackage{geometry}
\usepackage{hyperref}
\usepackage{mathtools}

\mathtoolsset{showonlyrefs}

\numberwithin{equation}{section}

\theoremstyle{plain}
\newtheorem{theorem}{Theorem}[section]
\newtheorem{proposition}[theorem]{Proposition}
\newtheorem{propdef}[theorem]{Proposition-Definition}
\newtheorem{lemma}[theorem]{Lemma}
\newtheorem{corollary}[theorem]{Corollary}
\theoremstyle{definition}
\newtheorem{definition}[theorem]{Definition}
\theoremstyle{remark}
\newtheorem{remark}[theorem]{Remark}

\newcommand{\R}{\mathbb{R}}

\renewcommand{\P}{\mathbb{P}}

\newcommand{\cover}{\mathfrak{U}}
\newcommand{\currs}{\mathcal{C}}

\newcommand{\holo}{\mathcal{O}}

\newcommand{\opens}{\mathcal{U}}
\newcommand{\sheafE}{\mathcal{E}}
\newcommand{\sheafF}{\mathcal{F}}

\renewcommand{\dbar}{\bar{\partial}}

\DeclareMathOperator{\codim}{codim}
\DeclareMathOperator{\End}{End}
\DeclareMathOperator{\Ext}{Ext}
\DeclareMathOperator{\Hom}{Hom}
\DeclareMathOperator{\Homs}{\mathcal{H}\!\mathit{om}}
\DeclareMathOperator{\id}{id}
\DeclareMathOperator{\im}{im}
\DeclareMathOperator{\supp}{supp}

\newcommand{\defeq}{\vcentcolon=}

\title[A residue current associated with a twisting cochain]{A residue current associated with a twisting cochain: duality and comparison formula}
\author{Jimmy Johansson}

\begin{document}
\begin{abstract}
Let $\sheafF^\bullet$ be a complex of coherent $\holo_X$-modules over a complex manifold $X$. We give a construction of a residue current associated with this complex that generalizes Andersson and Wulcan's construction of a residue current associated with a coherent $\holo_X$-module $\sheafF$. The main ingredients are to give a construction of a residue current associated with a twisting cochain in the sense of Toledo and Tong and to establish a comparison formula for these currents that generalizes an analogous formula given by Lärkäng.
\end{abstract}
\maketitle
\section{Introduction}
Let $X$ be a complex manifold, and let
\begin{equation}
\label{eq:bundle-complex}
	0 \longrightarrow F^{-N} \longrightarrow
	\dots \longrightarrow F^{-1} \longrightarrow
	F^0 \longrightarrow 0
\end{equation}
be a generically exact complex of holomorphic vector bundles over $X$ equipped with Hermitian metrics. In \cite{AW1}, Andersson and Wulcan gave a construction of an $(\End F)$-valued current $R^F$. It is referred to as the residue current associated with the complex $F^\bullet$. The main result that they proved is that if the corresponding complex of locally free sheaves is exact at each level $r < 0$, then $R^F$ has the property that a holomorphic section $\phi$ of $F^0$ belongs to $\im \left(F^{-1} \to F^0 \right)$ if and only if $R^F \phi = 0$. This is referred to as the \emph{duality principle}.

Andersson's and Wulcan's construction is a generalization of the classical Coleff--Herrera product introduced in \cite{CH}. These developments in multivariable residue theory has led to many results in commutative algebra and complex geometry. For example it has led to results about membership problems and Brian\c{c}on--Skoda type results, see, e.g, \cite{AW4}. It has led to results regarding the solution to the $\dbar$-equation on singular spaces, see, e.g., \cite{AS}. Moreover, it has provided explicit constructions of Noetherian operators as described in \cite{And1}.

Suppose that \eqref{eq:bundle-complex}, as a complex of locally free $\holo_X$-modules, is a locally free resolution of a coherent $\holo_X$-module $\sheafF$. The current $R^F$ is sometimes referred to as the residue current associated with the sheaf $\sheafF$ rather than the complex itself.
However, given a coherent $\holo_X$-module $\sheafF$, one is only guaranteed that locally one can find resolutions of locally free sheaves of $\sheafF$. See \cite{Voisin}*{Corollary~A.5} for an example of a coherent $\holo_X$-module which does not admit a locally free resolution. Thus unless one is in the setting where global resolutions of locally free sheaves exist, e.g. $X = \P^n$, the problems that can be approached using techniques stemming from the construction above are of a local nature.

Toledo and Tong showed in \cite{TT1} that one can assemble such local resolutions in an appropriate way to construct a global object known as a twisted resolution that in some sense serves as a substitute of a global resolution of locally free sheaves. For example, the authors showed that the global Ext functor $\Ext^k(\sheafF,\mathcal{G})$ can be computed as the cohomology of a complex involving a twisted resolution of $\sheafF$. In \cite{JL} we gave a generalization to Andersson's and Wulcan's construction to associate a residue current with a twisted resolution. Our main application was the following. If $\mathcal{G} = G$ is a holomorphic vector bundle, then it is well known that $\Ext^k(\mathcal{F},G)$ also can be computed as the $k$th cohomology of the complex $\Hom(\mathcal{F},\currs^{0,\bullet}(G))$, where $\currs^{0,q}(G)$ denotes the sheaf of $G$-valued $(0,q)$-currents. In \cite{JL}, we gave an explicit isomorphism between these two representations of the Ext functor using the residue current associated with a twisted resolution $(F,a)$, see Section~\ref{section:twisting}. This may be considered as a global version of a corresponding statement for the Ext sheaves, which is due to Andersson, \cite{And1}, generalizing earlier work by Dickenstein--Sessa, \cite{DS}.

A twisted resolution give rise to a complex $\left( \bigoplus_{p+r=\bullet} C^p(\cover,F^r), D \right)$, and in this paper we will show the associated residue current enjoys a similar duality principle with respect to this complex.
\begin{theorem}
\label{thm:twisted-duality}
Let $(F,a)$ be a twisted resolution of a coherent $\holo_X$-module $\sheafF$, and let $R$ be the associated residue current. Moreover, let $\phi$ be an element in $\bigoplus_{p+r=k} C^p(\cover,F^r)$. We have that $D \psi = \phi$ for some $\psi$ if and only if $D \phi = 0$ and $R \phi = 0$.
\end{theorem}

Let us return again to the residue currents when \eqref{eq:bundle-complex} is a locally free resolution of a coherent $\holo_X$-module $\sheafF$. The construction of $R^F$ depends both on the choice of locally free resolution as well as Hermtian metrics, but it is unique up to homotopy in the following sense. Suppose $E^\bullet$ is another locally free resolution of $\sheafF$. It is well known that one can find a chain map $\varphi: E^\bullet \to F^\bullet$ that commutes with the augmentation maps  $F^\bullet \to \sheafF$ and $E^\bullet \to \sheafF$. Lärkäng showed in \cite{Lar} that the associated residue currents $R^F$ and $R^E$ are related via
\begin{equation}
\label{eq:comparison}
	R^F \varphi - \varphi R^E = \nabla M,
\end{equation}
where $M$ is the current described in \cite{Lar}*{Theorem~1.2}, and $\nabla$ is the differential on the total complex of the double complex $\currs^{0,\bullet}(\Hom^\bullet(E,F))$, i.e., $\Hom(E,F)$-valued $(0,*)$-currents. The formula \eqref{eq:comparison} is referred to as the \emph{comparison formula}.

Twisted resolutions are instances of a concept known as \emph{twisting cochains}. The main aim of this paper is to show how one can associate a residue current to an arbitrary twisting cochain, and we will show that this residue current is unique up to a similar homotopy as in the local case. To do this we will give a generalization of Lärkäng's comparison formula.
\begin{theorem}
Let $\sheafF^\bullet$ and $\mathcal{E}^\bullet$ be complexes of coherent $\holo_X$-modules, and let $f: \sheafE^\bullet \to \sheafF^\bullet$ be a chain map. Moreover, let $(F,a)$ and $(E,b)$ be twisted resolutions of $\sheafF^\bullet$ and $\mathcal{E}^\bullet$ respectively, and let $R^F$ and $R^E$ be the associated residue currents. There exists a morphism $\varphi: (E,b) \to (F,a)$ of twisted resolutions, and we have the following homotopy:
\begin{equation}
\label{eq:Rhomotopy}
	R^F \varphi - \varphi R^E = \nabla M,
\end{equation}
where $M$ is defined by \eqref{eq:Mdef}.
\end{theorem}

The motivation for this construction comes from the following more general point of view.
If we view $\sheafF$ as a complex concentrated in degree 0, then we have a quasi-isomorphism $F^\bullet \to \sheafF$, i.e., a chain map that induces isomorphisms of the cohomology groups. This motivates the following generalization where we wish to associate a residue current with a complex of coherent $\holo_X$-modules
\[
	0 \longrightarrow \mathcal{F}^{-N} \longrightarrow
	\dots \longrightarrow \mathcal{F}^{-1} \longrightarrow
	\mathcal{F}^0 \longrightarrow 0.
\]
It is well known that locally one can find a complex of locally free $\holo_X$-modules and a quasi-isomorphism $F^\bullet \to \sheafF^\bullet$. In \cite{OTT2}, O'Brian, Toledo, and Tong generalized the notion of twisted resolutions to this more general situation. In this case it is no longer true that the complexes of sheaves are generically exact, but we will show that with minor modifications the same construction that we used in \cite{JL} will yield a well-defined current that we shall refer to as the residue current associated with $\sheafF^\bullet$. As mentioned, we will show that this construction is independent of twisted resolution up to a certain homotopy in the following sense.
\begin{theorem}
\label{thm:Rhomotopy}
Let $\sheafF^\bullet$ be a complex of coherent $\holo_X$-modules. Let $(F,a)$ and $(E,b)$ be twisted resolutions, and let $R^F$ and $R^E$ be the associated residue currents. There exist morphisms $\varphi: (E,b) \to (F,a)$ and $\psi: (F,a) \to (E,b)$ such that $\varphi \psi$ is homotopic to the identity, i.e., $\varphi \psi - \id = \nabla \alpha$, and $R^F$ is homotopic to $\varphi R^E \psi$. The homotopy is given by
\[
	R^F - \varphi R^E \psi = \nabla(M \psi - R^F \alpha),
\]
where $M$ is defined by \eqref{eq:Mdef}.
Analogously, we have that $R^E$ is homotopic to $\psi R^F \varphi$ via a similar formula.
\end{theorem}

This paper is organized as follows. Residue currents arise as residues of almost semimeromorphic currents and they belong to a class known as pseudomeromorphic currents. We shall begin in Section~\ref{section:pseudo} by recalling the necessary definitions and results about these currents. In Section~\ref{section:twisting} we will give a brief introduction to twisting cocycles and twisted resolutions. In Section~\ref{section:residue} we describe the construction of a residue current associated with a twisting cochain. In Section~\ref{section:comparison} we establish a comparison formula for these currents. Finally, in Section~\ref{section:morphisms} we describe how one can associate a residue current with a bounded complex of coherent $\holo_X$-modules and make precise its uniqueness up to homotopy.
\section{Pseudomeromorphic and almost semimeromorphic currents}
\label{section:pseudo}
Let $s$ be a holomorphic section of a Hermitian holomorphic line bundle $L$ over $X$.
The \emph{principal value current} $[1/s]$ can be defined as $[1/s] \defeq \lim_{\epsilon \to 0} \chi(|s|^2/\epsilon)\frac{1}{s}$,
where $\chi : \R \to \R$ is a smooth cut-off function, i.e., $\chi(t) = 0$ in a neighborhood of zero and $\chi(t) = 1$ when $|t| \gg 1$.
A current is \emph{semimeromorphic} if it is of the form $[\omega/s] \defeq \omega[1/s]$, where $\omega$ is a smooth form with values in $L$.
A current $a$ is \emph{almost semimeromorphic} on $X$, written $a \in ASM(X)$, if there is a modification $\pi : X' \to X$ such that
$$a=\pi_*(\omega/s),$$ where $\omega/s$ is a semimeromorphic current in $X'$.
The almost semimeromorphic currents on $X$ form an algebra over
smooth forms.

Almost semimeromorphic functions are special cases of so-called pseudomeromorphic currents, as introduced in \cite{AW2}.
The class of pseudomeromorphic currents is closed under multiplication with smooth forms and under $\dbar$.
One important property of pseudomeromorphic currents is that they satisfy the following \emph{dimension principle}.

\begin{proposition} \label{prop:dimPrinciple}
Let $T$ be a pseudomeromorphic $(*,q)$-current on $X$ with support on a subvariety $Z$.
If $\codim Z > q$, then $T = 0$.
\end{proposition}

Given a pseudomeromorphic current $T$ and a subvariety $V$, in \cite{AW2} was introduced a restriction of $T$ to $X\setminus V$,
which is a pseudomeromorphic current on $X$ defined by $\mathbf{1}_{X \setminus V} T \defeq \lim_{\epsilon \to 0} \chi(|F|^2/\epsilon) T$, where $\chi$
is a cut-off function as above, and $F$ is a section of a holomorphic vector bundle such that $V = \{ F = 0 \}$.
A pseudomeromorphic current $T$ on $X$ is said to have the standard extension property (SEP) if $\mathbf{1}_{X \setminus V} T = T$
for any subvariety $V$ of positive codimension.
It follows from the dimension principle and the fact that the restriction commutes with multiplication with smooth forms that
almost semimeromorphic currents have the SEP. In particular, if $\alpha$ is a smooth form on $X \setminus V$, and $\alpha$ has
an extension as an almost semimeromorphic current $a$ on $X$, then the extension is given by
\begin{equation} \label{eq:asmExtension}
    a = \lim_{\epsilon \to 0} \chi(|F|^2/\epsilon) \alpha.
\end{equation}

Let $A$ be an almost semimeromorphic current on $X$. Let $Z$ be the smallest subvariety of $X$ of positive codimension such that $a$
is smooth outside of $Z$. By \cite{AW3}*{Proposition~4.16}, $\mathbf{1}_{X \setminus Z} \dbar A$ is almost semimeromorphic on $X$.
The \emph{residue} $R(A)$ of $A$ is defined by
\begin{equation}
\label{eq:residue-def}
	R(A) \defeq \dbar A - \mathbf{1}_{X \setminus Z}\dbar A.
\end{equation}
Note that
\begin{equation} \label{eq:resSupport}
    \supp R(A) \subseteq Z.
\end{equation}
Since $A$ is almost semimeromorphic, and thus has the SEP, it follows
by \eqref{eq:asmExtension} that
\begin{equation}
\label{eq:residue}
    R(A)=\lim_{\epsilon \to 0}
    \left(\dbar(\chi_\epsilon A) - \chi_\epsilon \dbar A \right)
    = \lim_{\epsilon \to 0} \dbar\chi_\epsilon \wedge A,
\end{equation}
where $\chi$ is as above, $F$ is a section of a vector bundle such that $\{ F = 0 \} \supseteq Z$, and $F \not\equiv 0$
and $\chi_\epsilon = \chi(|F|^2/\epsilon)$.
It follows directly from for example \eqref{eq:residue} that if $\psi$ is a smooth form, then
\begin{equation} \label{eq:residueSmooth}
    R(\psi \wedge a) = (-1)^{\deg \psi} \psi\wedge R(a).
\end{equation}
\section{Twisting cochains}
\label{section:twisting}
Throughout this paper, let $X$ be a complex manifold of dimension $n$, and let $\cover = (\opens_\alpha)$ be a covering of $X$ by open sets. Unless otherwise specified, we will assume these to be Stein. This condition is not necessary for the definitions that we shall make in this section, but it will be used to prove various results in the next section.
We will use the notation $\opens_{\alpha_0 \dots \alpha_p} \defeq \opens_{\alpha_0} \cap \dots \cap \opens_{\alpha_p}$.

In this section, we will mainly recall the relevant parts about twisting cochains, twisted resolutions and twisted complexes from \cite{TT1}.
This material is essentially the same as described in Sections 1 and 2 of \cite{TT1}.
As in \cite{JL}, we will also incorporate this theory with the theory of residue currents from \cite{AW1}. To get our sign convention consistent
with both \cite{TT1} and \cite{AW1}, we consider current-valued analogues of the objects described in \cite{TT1}.

Throughout this section we let $E \defeq (E_\alpha)$, $F \defeq (F_\alpha)$, and $G \defeq (G_\alpha)$ be families of bounded graded holomorphic vector bundles over $\opens_\alpha$. Recall that
\[
	\Hom^r(E_\alpha,F_\beta) =
	\bigoplus_j
	\Hom(E_\alpha^{-j},F_\beta^{-j+r}).
\]

Let $\currs^{0,q}(\Hom^r(E_\alpha,F_\beta))$ denote the sheaf over $\opens_{\alpha \beta}$ of $\Hom^r(E_\alpha,F_\beta)$-valued $(0,q)$-currents. If $f$ is a section of $\currs^{0,q}(\Hom^r(F_\beta,G_\gamma))$ and $g$ is a section of $\currs^{0,q'}(\Hom^{r'}(E_\alpha,F_\beta))$, then we can define their product as a section of $\currs^{0,q+q'}(\Hom^{r+r'}(E_\alpha,G_\gamma))$ given by
\begin{equation}
\label{eq:prod1}
	fg \defeq
	(-1)^{rq'}
	f \wedge g
\end{equation}
provided that the wedge product of $f$ and $g$ is defined. By identifying $E_\alpha$ with $\Hom(\holo_X,E_\alpha)$, where $\holo_X$ is viewed as a graded vector bundle concentrated in degree $0$, we can also define the product of sections of $\currs^{0,q}(\Hom^r(E_\alpha,F_\beta))$ and $\currs^{0,q'}(E_\alpha)$.

We will consider a sort of \v{C}ech cochains with coefficients in these sheaves. We define
\begin{equation}
\label{eq:cechCurrentHom}
	C^p(\cover,\currs^{0,q}(\Hom^r(E,F))) \defeq
	\prod_{(\alpha_0, \dots, \alpha_p)}
	\currs^{0,q}(\Hom^r(E_{\alpha_p},F_{\alpha_0}))
	({\opens_{\alpha_0 \dots \alpha_p}}),
\end{equation}
and
\begin{equation}
\label{eq:cechCurrentVec}
	C^p(\cover,\currs^{0,q}(F^r)) \defeq
	C^p(\cover,\currs^{0,q}(\Hom^r(\holo_X,F))).
\end{equation}
For an element $f \in C^p(\cover,\currs^{0,q}(\Hom^r(E,F)))$, we define its degree as $\deg f = p+q+r$,
and we call $p$ the \v{C}ech degree, $q$ the current degree, and $r$ the Hom degree.

For elements $f \in C^p(\cover,\currs^{0,q}(\Hom^r(F,G)))$ and $g \in C^{p'}(\cover,\currs^{0,q'}(\Hom^{r'}(E,F)))$ we define their product $fg \in C^{p+p'}(\cover,\currs^{0,q+q'}(\Hom^{r+r'}(E,G)))$ by
\begin{equation}
\label{eq:prod2}
	(fg)_{\alpha_0 \dots \alpha_{p+p'}} \defeq
	(-1)^{(q+r)p'}
	f_{\alpha_0 \dots \alpha_p} g_{\alpha_p \dots \alpha_{p+p'}},
\end{equation}
where the product on the right-hand side is defined by \eqref{eq:prod1} provided that it exists.

We let the $\dbar$-operator act as an operator of degree 1 on $C^\bullet(\cover,\currs^{0,\bullet}(\Hom^\bullet(E,F)))$ by
\begin{equation}
\label{eq:dbar-cech}
	(\dbar f )_{\alpha_0 \dots \alpha_p} \defeq
	(-1)^p \dbar f_{\alpha_0 \dots \alpha_p}.
\end{equation}
With this definition we have that $\dbar(fg) = (\dbar f)g + (-1)^{\deg f} f (\dbar g$), and as usual $\dbar^2 = 0$.

Next we define an operator of degree 1,
\[
	\delta: C^p(\cover,\currs^{0,q}(\Hom^r(E,F))) \to
	C^{p+1}(\cover,\currs^{0,q}(\Hom^r(E,F))),
\]
by
\[
	(\delta f)_{\alpha_0 \dots \alpha_{p+1}} \defeq
	\sum_{k=1}^p (-1)^k
	f_{\alpha_0 \dots \widehat{\alpha}_k \dots \alpha_{p+1}}
	|_{\opens_{\alpha_0 \dots \alpha_{p+1}}}.
\]
Note that $\delta$ is similar to the usual \v{C}ech coboundary, but in the sum, it is necessary to omit $f_{\alpha_1 \dots \alpha_{p+1}}$ and $f_{\alpha_0 \dots \alpha_p}$ since these do not belong to $\Hom^r(F_{\alpha_{p+1}},F_{\alpha_0})$. However, we still have that $\delta$ is a differential and an antiderivation with respect to the product \eqref{eq:prod2}, i.e., $\delta^2 = 0$, and
\begin{equation}
\label{eq:deltader}
	\delta(fg) = (\delta f) g + (-1)^{\deg f} f (\delta g).
\end{equation}

We are now ready to define the notion of a twisting cochain. We define
\[
	C^p(\cover,\Hom^r(E,F)) \defeq
	\prod_{(\alpha_0, \dots, \alpha_p)}
	\Homs^r(E_{\alpha_p},F_{\alpha_0})
	({\opens_{\alpha_0 \dots \alpha_p}}).
\]
Note that this is the subgroup of $\dbar$-closed elements of $C^p(\cover,\currs^{0,0}(\Hom^r(E,F)))$.
This subgroup is the group of \v{C}ech cochains considered in \cite{TT1}, and when we restrict to this subgroup, our sign convention is consistent with \cite{TT1}. Analogous with the situation above, we also make the definition
\[
	C^p(\cover,F^r) \defeq
	C^p(\cover,\Hom^r(\holo_X,F)).
\]
\begin{definition}
\label{def:twisting-cochain}
A \emph{twisting cochain} $a \in C^\bullet(\cover,\Hom^\bullet(F,F))$ is an element $a = \sum_{k \geq 0} a^k$ of degree 1, where $a^k \in  C^k(\cover,\Hom^{1-k}(F,F))$,
such that
\begin{equation}
\label{eq:twisting-cochain}
	\delta a + aa = 0,
\end{equation}
and $a_{\alpha \alpha}^1 = \id_{F_\alpha}$ for all $\alpha$. For simplicity we shall simply refer to the pair $(F,a)$ as a twisting cochain.
\end{definition}

Note that \eqref{eq:twisting-cochain} is equivalent to
\begin{equation}
\label{eq:twisting-cochain2}
	\delta a^{k-1} + \sum_{j=0}^k a^j a^{k-j} = 0,
	\quad k \geq 0.
\end{equation}
Explicitly, the first three equations are
\begin{align}
	&a_\alpha^0 a_\alpha^0 = 0 \label{eq:twisted0} \\
	&a_\alpha^0 a_{\alpha \beta}^1 = a_{\alpha \beta}^1 a_\beta^0 \label{eq:twisted1} \\
	&a_{\alpha \gamma}^1 - a_{\alpha \beta}^1 a_{\beta \gamma}^1 =
	a_\alpha^0 a_{\alpha \beta \gamma}^2 + a_{\alpha \beta \gamma}^2 a_\gamma^0 \label{eq:twisted2}.
\end{align}
The first equation says that $(F_\alpha,a_\alpha^0)$ is a complex, the second says that $a_{\alpha \beta}^1$ defines a chain map $(F_\beta|_{\opens_{\alpha \beta}},a_\beta^0) \to (F_\alpha|_{\opens_{\alpha \beta}},a_\alpha^0)$, and the third says that, over $\opens_{\alpha \beta \gamma}$, $a_{\alpha \gamma}^1$ and $a_{\alpha \beta}^1 a_{\beta \gamma}^1$ are chain homotopic, with the homotopy given by $a^2_{\alpha\beta\gamma}$. In particular, from the condition $a_{\alpha \alpha}^1 = \id_{F_\alpha}$, it follows that $a_{\alpha \beta}^1$ and $a_{\beta \alpha}^1$ are chain homotopy inverses to each other.

Consider two twisting cochains $(F,a)$ and $(E,b)$. We define an operator $D$ on $C^\bullet(\cover,\currs^{0,\bullet}(\Hom(E,F)))$,
\begin{equation}
\label{eq:D}
	D \varphi \defeq
	\delta \varphi + a \varphi - (-1)^{\deg \varphi} \varphi b.
\end{equation}
It can be shown that $D^2=0$, and
\[
	D(fg) = (Df)g + (-1)^{\deg f} f (Dg).
\]

We can now give a definition of a morphism between twisting cochains. This notion was first introduced in \cite{Gil}*{Definition~3.11}.
\begin{definition}
Let $(F,a)$ and $(E,b)$ be twisting cochains. A \emph{morphism of twisting cochains} $\varphi: (E,b) \to (F,a)$ is a $D$-closed element $\varphi \in C^\bullet(\cover,\Hom(E,F))$ of degree 0.
\end{definition}
The twisting cochains that are of main interest to us arise as so-called \emph{twisted resolutions} of coherent $\holo_X$-modules, see \cites{TT1,OTT}, or more generally complexes of coherent $\holo_X$-modules, see \cite{OTT2}.
\begin{definition}
Let $\sheafF^\bullet$ be a bounded complex of coherent $\holo_X$-modules. A \emph{twisted resolution} of $\sheafF^\bullet$ consists of a twisting cochain $(F,a)$ such that for each $\alpha$ there is a quasi-isomorphism $F_\alpha^\bullet \to \sheafF^\bullet|_{\opens_\alpha}$.
\end{definition}
Let $\sheafF^\bullet$ be a bounded complex of $\holo_X$-modules. It is well known that there exists a cover by Stein open sets such that for each $\opens_\alpha$ there exists a complex $(F_\alpha^\bullet,a_\alpha^0)$ that is quasi-isomorphic with $\sheafF^\bullet|_{\opens_\alpha}$, see, e.g., the method outlined in \cite{Eis}*{Exercise~3.53}. It was shown in \cite{OTT2}*{Proposition~1.2.3} that $a^0$ can be extended to a twisted resolution of $\sheafF^\bullet$.

We combine $D$ and $\dbar$ into an operator
\[
	\nabla = D - \dbar
\]
of degree 1 on $C^\bullet(\cover,\currs^{0,\bullet}(\Hom^\bullet(E,F)))$. It can be shown that $\dbar D = -D \dbar$, and from this it follows that $\nabla^2 = 0$. Moreover, we have that
\begin{equation}
\label{eq:nablaDerivation}
	\nabla(fg) = (\nabla f) g + (-1)^{\deg f} f (\nabla g).
\end{equation}

We end this section by introducing some additional notation that will be useful in the sequel. Let $\cover' = (\opens'_\alpha)$ be such that $\opens'_\alpha \subseteq \opens_\alpha$ for each $\alpha$. We have a homomorphism
\begin{align*}
	C^\bullet(\cover,\currs^{0,\bullet}(\Hom(E,F))) &\to
	C^\bullet(\cover',\currs^{0,\bullet}(\Hom(E,F))) \\
	A &\mapsto A|_{\cover'}
\end{align*}
induced by the restriction maps. We have that this homomorphism commutes with $\dbar$, $D$, and $\nabla$, as well as the multiplication \eqref{eq:prod2}.

Let $f \in C^p(\cover,\currs^{0,q}(\Hom^r(E,F)))$.  We denote by $f_k^\ell$ the part of $f$ that belong to
\[
	\prod_{(\alpha_0, \dots, \alpha_p)}
	\currs^{0,q}(\Hom(E_{\alpha_p}^{-\ell},F_{\alpha_0}^{-k}))
	(\opens_{\alpha_0 \dots \alpha_p}),
\]
i.e., when using both a superscript and a subscript, we pick out morphisms between bundles
in certain degrees. We will also use the notation
\[
	f^\ell_\bullet \defeq \sum_k f^\ell_k.
\]
We say that \emph{$f$ takes values in $\Hom(F^{-\ell},F^{-k})$} if $f=f^\ell_k$,
and that \emph{$f$ takes values in $\Hom(F^{-\ell},F)$} if $f=f^\ell_\bullet$.
\subsection{Residues of cochains with coefficients in $ASM(X)$}
For elements in \eqref{eq:cechCurrentHom}, we will say that they are almost semimeromorphic and pseudomeromorphic respectively if their components are. We will now extend the notion of residues to almost semimeromorphic elements of $C^\bullet(\cover,\currs^{0,\bullet}(\Hom(E,F)))$. 

For each $\alpha$ let $Z_\alpha$ be a subvariety of $\opens_\alpha$. Let $\opens'_\alpha \defeq \opens_\alpha \setminus Z_\alpha$, and define the family of open sets $\cover' \defeq (\opens'_\alpha)$. Let $T \in C^\bullet(\cover,\currs^{0,\bullet}(\Hom(E,F)))$ be pseudomeromorphic. We define $\mathbf{1}_{\cover'} T$ by
\[
	(\mathbf{1}_{\cover'} T)_{\alpha_0 \dots \alpha_p} =
	\mathbf{1}_{\opens'_{\alpha_0 \dots \alpha_p}} T_{\alpha_0 \dots \alpha_p}.
\]
Suppose $u \in C^\bullet(\cover',\currs^{0,\bullet}(\Hom(E,F)))$ is smooth. For each $(\alpha_0,\dots,\alpha_p)$ let $s_{\alpha_0 \dots \alpha_p} \neq 0$ be a section of a vector bundle over $\opens_{\alpha_0 \dots \alpha_p}$ whose zero locus contains $Z_{\alpha_0} \cup \dots \cup Z_{\alpha_p}$. If $u$ has an extension as an almost semimeromorphic current $U \in C^\bullet(\cover,\currs^{0,\bullet}(\Hom(E,F)))$, then, as in \eqref{eq:asmExtension} it is given by
\[
	U_{\alpha_0 \dots \alpha_p} \defeq
	\lim_{\epsilon \to 0}
	\chi(|s_{\alpha_0 \dots \alpha_p}|^2/\epsilon)
	u_{\alpha_0 \dots \alpha_p}.
\]

For an almost semimeromorphic current $A \in C^\bullet(\cover,\currs^{0,\bullet}(\Hom(E,F)))$, in view of \eqref{eq:dbar-cech}, we define its residue $R(A)$ by
\[
	R(A)_{\alpha_0 \dots \alpha_p} \defeq (-1)^p R(A_{\alpha_0 \dots \alpha_p}).
\]
For each $\alpha$, suppose that there exists a subvariety $Z_\alpha$ such that $A_{\alpha_0 \dots \alpha_p}$ is smooth outside of $Z_{\alpha_0} \cup \dots \cup Z_{\alpha_p}$. If $\cover' = (\opens_\alpha \setminus Z_\alpha)$, then we have that
\[
	R(A) = \dbar A - \mathbf{1}_{\cover'} \dbar A,
\]
in analogy with \eqref{eq:residue-def}. We also have the following useful formula which is analogous to (2.4) in \cite{Lar}. Suppose that $B \in C^\bullet(\cover,\currs^{0,\bullet}(\Hom(E,F)))$ is almost semimeromorphic, and that $\nabla (A|_{\cover'}) = B|_{\cover'}$. Then
\begin{equation}
\label{eq:nabla-A-B}
	R(A) = B - \nabla A.
\end{equation}
Indeed, we have that $\dbar (A|_{\cover'}) = (DA - B)|_{\cover'}$. Since $DA$ is almost semimeromorphic and therefore has the SEP, it follows that $\mathbf{1}_{\cover'} \dbar A = DA - B$. Then
\[
	R(A) =
	\dbar A - \mathbf{1}_{\cover'} \dbar A =
	\dbar A - DA + B = B - \nabla A.
\]
\section{A residue current associated with a twisting cochain}
\label{section:residue}
We will now give the construction of a residue current associated with an arbitrary twisting cochain. Throughout this section, let $(F,a)$ be a twisting cochain on a complex manifold $X$. We will tacitly assume that the bundles $F_\alpha^k$ are equipped with Hermitian metrics. As usual we will denote the cover of $X$ by $\cover \defeq (\opens_\alpha)$.

For each $\alpha$ let $Z_\alpha$ denote the subvariety of $\opens_\alpha$ where $a_\alpha^0$ does not have optimal rank. Let $\opens'_\alpha \defeq \opens_\alpha \setminus Z_\alpha$, and define the cover $\cover' \defeq (\opens'_\alpha)$.
We define an element $\sigma^0 \in C^0(\cover',\currs^{0,0}(\Hom^{-1}(F,F)))$ in the following way.
Let $\sigma_\alpha^0$ be the minimal right-inverse of $a_\alpha^0$ on $\opens'_\alpha$, i.e., the Moore--Penrose inverse. It satisfies the properties $a_\alpha^0 \sigma_\alpha^0 a_\alpha^0 = a_\alpha^0$, $\sigma_\alpha^0|_{(\im a_\alpha^0)^\perp}= 0$, and $\im \sigma_\alpha^0 \perp \ker a_\alpha^0$. From the last two properties it follows that $(\sigma^0)^2 = 0$.
We write
\[
	a = a^0 + a' \quad \text{where} \quad a' \defeq \sum_{k\geq 1} a^k.
\]
Define
\begin{equation} \label{eq:sigmadef}
    \sigma \defeq
    \sigma^0(\id+a' \sigma^0)^{-1} =
    \sigma^0 - \sigma^0 a' \sigma^0 +
    \sigma^0 a' \sigma^0 a' \sigma^0 - \dots.
\end{equation}
Note that this is a well-defined element of degree $-1$ since over each $\opens_{\alpha_0 \dots \alpha_p}$, the right-hand is finite since $a'\sigma^0$ has negative Hom degree and the complexes have finite length.
Since $(\sigma^0)^2 = 0$, it follows that $\sigma^2 = 0$.

\begin{propdef}
\label{propdef:U-Rdef}
Define
\begin{equation}
\label{eq:udef}
    u \defeq \sigma(\id-\dbar\sigma)^{-1} =
    \sigma+\sigma\dbar\sigma+\sigma(\dbar\sigma)^2 + \dots,
\end{equation}
which has degree $-1$. Then $u$ has an almost semimeromorphic extension
\[
	U \in \bigoplus_{p+q+r=-1}C^p(\cover,\currs^{0,q}(\Hom^r(F,F))).
\]
Moreover, we define
\begin{equation} \label{eq:Rdef}
	R \defeq \id - \nabla U,
\end{equation}
which has degree $0$, is pseudomeromorphic and $\nabla$-closed, i.e., $\nabla R = 0$. We also have the decomposition
\begin{equation}
\label{eq:Rcomp}
	R = R' + R(U),
\end{equation}
where $R'$ is almost semimeromorphic and $R'|_{\cover'}$ is smooth. If each complex $F_\alpha^\bullet$ is generically exact, then $R' = 0$. For the residue part we have that 
\begin{equation}
\label{eq:Rsupp}
	R(U)|_{\cover'} = 0.
\end{equation}
\end{propdef}
\begin{proof}
The existence of an almost semimeromorphic extension follows by the same arguments as in \cite{JL}*{Proposition-Definition 4.1}, and the decomposition \eqref{eq:Rcomp} follows using the same argument as in the first part of \cite{JL}*{Proposition~4.4}.

If each complex $F_\alpha^\bullet$ is generically exact, then $a^0 \sigma^0 + \sigma^0 a^0 = \id$, and we get that $R' = 0$ by the arguments in Lemma~4.5 and Lemma~4.6 in \cite{JL}.
\end{proof}

We shall refer to $R$ as the residue current associated with the twisting cochain $(F,a)$. It is related to the exactness of the complex $\left( \bigoplus_{p+r=\bullet} C^p(\cover,F^r), D \right)$ in the following way. This result generalizes \cite{AW1}*{Proposition~2.3}, see Remark~\ref{rmk:aw}.
\begin{proposition}
\label{prop:twisted-duality}
Let $\phi$ be an element in $\bigoplus_{p+r=k} C^p(\cover,F^r)$.
\begin{enumerate}[(i)]
\item
If $D \phi = 0$ and $R \phi = 0$, then there exists an element $\psi \in \bigoplus_{p+r=k-1} C^p(\cover,F^r)$ such that $D \psi = \phi$.
\item
Conversely, if $\phi = D \psi$ for some $\psi \in \bigoplus_{p+r=k-1} C^p(\cover,F^r)$ and $R_\bullet^\ell = 0$ for $\ell > -k$, then $R \phi = 0$.
\end{enumerate}
\end{proposition}
\begin{proof}
(i) Assume that $D \phi = \nabla \phi = 0$ and $R \phi = 0$. Then we have that
\[
	\nabla (U \phi) =
	(\nabla U) \phi - U (\nabla \phi) =
	\phi - R \phi = \phi,
\]
where the second equality follows by \eqref{eq:Rdef}. Thus $v = U \phi \in \bigoplus_{p+q+r=k-1} C^p(\cover,\currs^{0,q}(F^r))$ is a solution to the system of equations,
\[
	D v^0 = \phi, \quad
	D v^\ell = \dbar v^{\ell-1}, \quad
	\ell \geq 1,
\]
where $v^q$ denotes the part of $v$ of bidegree $(0,q)$. Since $\cover$ is Stein there exists an element $w \in \bigoplus_{p+q+r=k-2} C^p(\cover,\currs^{0,q}(F^r))$ that satisfies the system $\dbar$-equations
\[
	\dbar w^\ell = v^{\ell+1} + D w^{\ell+1},
\]
where $w^q$ denotes the part of $w$ of bidegree $(0,q)$.
Let $\psi = v^0 + D w^0$. Then $\psi$ is an element of $\bigoplus_{p+r=k-1} C^p(\cover,F^r)$ since $\dbar \psi = 0$, and, moreover, we have that $D \psi = \phi$.

(ii) Since $\nabla R = 0$, we have that $R \phi = R\nabla \psi = \nabla(R \psi)$. By the assumption on $R$ and the fact that the components of $\psi$ take values in $F^{-\ell}$, $\ell > -k$ it follows that $R \psi = 0$. From this the statement follows.
\end{proof}
\begin{remark}
\label{rmk:aw}
The currents that we construct in this paper are generalizations of the currents $U$ and $R$ constructed in \cite{AW1}. Let $X$ be a complex manifold that is not necessarily Stein. Suppose that we have a complex of the form \eqref{eq:bundle-complex}. This defines a twisting cochain in the following way. We take $\cover$ to be the cover consisting of the single open set $\opens_\alpha \defeq X$. We take $a^k = 0$ for $k \geq 1$, and we define $a^0 = a_\alpha^0$ to be the differential of the complex \eqref{eq:bundle-complex}. Then the complex \eqref{eq:bundle-complex} together with $a$ defines a twisting cochain.

In this case we get that $\sigma = \sigma_\alpha^0$ is the minimal right-inverse of $a_\alpha^0$, and the current $U$ coincides with the definition given in \cite{AW1}. Since the complex is generically exact we have that $R' = 0$, and hence $R = R(U)$ which coincides with the definition in \cite{AW1}.

In this setting Proposition~\ref{prop:twisted-duality} becomes precisely the statement of \cite{AW1}*{Proposition~2.3}: Let $\phi$ be a holomorphic section of $F^k$. If $a \phi = 0$ and $R \phi = 0$, then locally there exists a holomorphic section $\psi$ of $F^{k-1}$ such that $a \psi = \phi$. Note that we only have existence of local solution since we have made no assumptions on $X$ and the proof of Proposition~\ref{prop:twisted-duality} requires us to solve a sequence of $\dbar$-equations. Conversely, if $\phi = a \psi$ for some holomorphic section $\psi$ of $F^{k-1}$ and $R_\bullet^\ell = 0$ for $\ell > -k$, then $R \phi = 0$.

In addition to that we in this paper consider currents associated with families of locally free resolutions over a cover, we also make the generalization that we do not require the complexes involved to be generically exact. This introduces the additional term $R'$ that is not present in \cite{AW1}. The main reason for this is that we want to define residue currents associated with resolutions of chain complexes, and the locally free resolutions involved in this twisting cochain are in general not generically exact.
\end{remark}
We shall now give a more thorough description of the residue part of the current $R$. This will be done under the assumption that $(F,a)$ has what we shall refer to as \emph{singularity subvarieties}. For each $\alpha$, let $Z_\alpha^k \subseteq \opens_\alpha$ be the subvariety where $F_\alpha^{-k} \to F_\alpha^{-k+1}$ does not have optimal rank. If $Z^k \defeq \bigcup_\alpha Z_\alpha^k$ is a subvariety of $X$, then we say that $(F,a)$ has singularity subvarieties $Z^k$.
If $(F,a)$ has singularity subvarieties, then we have the following result on which parts of the residue part of $R$ that vanishes. This is analogous to \cite{Lar}*{Lemma~3.5} which in turn is a generalization of \cite{AW1}*{Theorem~3.1}.
\begin{proposition}
\label{prop:R-vanish}
Let $(F,a)$ be a twisting cochain, and let $U$ be the current as defined in Section~\ref{section:residue}. Suppose that $(F,a)$ has singularity subvarieties $Z^k$ which satisfy $\codim Z^{\ell+m} \geq m + 1$ for $m = 1, \dots, k - \ell$. Then $R(U)_k^\ell = 0$.
\end{proposition}
\begin{proof}
Since $\sigma$ has negative Hom degree, we need to prove that $R \left(\sigma (\dbar \sigma)^q \right)_k^\ell = 0$ for $q = 0, \dots, k-\ell-1$. By applying $\dbar$ to the equality $\sigma \sigma = 0$ it follows that $\sigma (\dbar \sigma)^q = (\dbar \sigma)^q \sigma$, and we may thus equivalently prove that
\begin{equation}
\label{eq:Rq}
	R \left( (\dbar \sigma)^q \sigma \right)_k^\ell = 0.
\end{equation}
We claim that $\eqref{eq:Rq}$ holds for $q = 0$. Indeed, $\sigma$ is smooth outside $\bigcup_{j+1 \leq m \leq k} Z^m$, and hence $R(\sigma)_k^\ell$ has support on this set. Since this set has codimension at least 2 and $R(\sigma)_k^\ell$ has bidegree $(0,1)$, we have that $R(\sigma)_k^\ell = 0$ by the dimension principle.

For $q > 0$ we will prove \eqref{eq:Rq} by induction over $k$.
To this end, assume that $R \left( (\dbar \sigma)^q \sigma \right)_j^\ell = 0$ for $j = \ell+1, \dots, k-1$ and $q = 0, \dots, j-\ell-1$.
We have that
\[
	R \left(
	(\dbar \sigma)^q \sigma
	\right)_k^\ell =
	\sum_{j=\ell+q}^{k-1}
	R \left(
	(\dbar \sigma)_k^j
	\left(
	(\dbar \sigma)^{q-1} \sigma
	\right)_j^\ell
	\right)
\]
for $q = 1, \dots, k-\ell-1$.
Outside $\bigcup_{j+1 \leq m \leq k} Z^m$ we have that $(\dbar \sigma)_k^j$ is smooth. Thus by \eqref{eq:residueSmooth} we have that
\[
	R \left( (\dbar \sigma)_k^j \left((\dbar \sigma)^{q-1} \sigma \right)_j^\ell \right) =
	(\dbar \sigma)_k^j R \left( (\dbar \sigma)^{q-1} \sigma \right)_j^\ell,
\]
which vanishes by the induction hypothesis. Thus $R \left( (\dbar \sigma)_k^j \left((\dbar \sigma)^{q-1} \sigma \right)_j^\ell \right)$ has support on $\bigcup_{j+1 \leq m \leq k} Z^m$, which has codimension at least $j-\ell+2$. Since it has bidegree at most $(0,j-\ell+1)$, it must vanish by the dimension principle.
\end{proof}
In the case that $(F,a)$ is a twisted resolution of a coherent $\holo_X$-module $\mathcal{F}$, then we have the following application of Proposition~\ref{prop:R-vanish}. In \cite{JL} we gave a direct proof of this result.
\begin{corollary}
\label{cor:Rvanish}
Let $(F,a)$ be a twisted resolution of a coherent $\holo_X$-module $\mathcal{F}$, and let $U$ be the current defined as in Proposition-Definition~\ref{propdef:U-Rdef}. Then $R(U)$ takes values in $\Hom(F^0,F)$ and $R(U)_k^0 = 0$ for $k < \codim \mathcal{F}$. Moreover, $\supp R \subseteq \supp \sheafF$.
\end{corollary}
\begin{proof}
By \cite{Eis}*{Theorem~20.9},
\[
    \codim Z_\alpha^k \geq k,
\]
and by \cite{Eis}*{Corollary 20.12},
\[
    Z_\alpha^{k+1} \subseteq Z_\alpha^k.
\]
Now the statement follows by Propositon~\ref{prop:R-vanish}.
\end{proof}
\begin{proof}[Proof of Theorem~\ref{thm:twisted-duality}]
By \cite{JL}*{Proposition~4.4} we have that $R'$ takes values in $\Hom(F^0,F)$, so that $R_k^\ell = 0$ for $\ell > 0$. Thus by this and Corollary~\ref{cor:Rvanish} we get that $R$ takes values in $\Hom(F^0,F)$. Now the statement follows by Proposition~\ref{prop:twisted-duality}.
\end{proof}
In the setting of Remark~\ref{rmk:aw}, then Theorem~\ref{thm:twisted-duality} is precisely the duality principle given in the introduction. Note we can drop the requirement that the complex is generically surjective at level 0 thanks to the component $R'$.
\section{A comparison formula}
\label{section:comparison}
We will now present a comparison formula for residue currents associated with twisting cochains. This is analogous to the comparison formula in \cite{Lar}. In this section we will consider residue currents associated with twisting cochains $(F,a)$ and $(E,b)$, and we will write $U^F$, $R^F$ and $U^E$, $R^E$ respectively for the associated currents defined as in Proposition-Definition~\ref{propdef:U-Rdef}.
\begin{theorem}
\label{thm:comparison}
Let $(F,a)$ and $(E,b)$ be twisting cochains, let $R^F$ and $R^E$ be the associated residue currents, and let $\varphi: (E,b) \to (F,a)$ be a morphism. Set
\begin{equation}
\label{eq:M'}
	M' \defeq (R^F)' \varphi U^E - U^F \varphi (R^E)',
\end{equation}
and let $M$ be defined as
\begin{equation}
\label{eq:Mdef}
	M \defeq M' + R(U^F \varphi U^E).
\end{equation}
Then
\[
	R^F \varphi - \varphi R^E = \nabla M.
\]
Moreover, if each complex $F_\alpha^\bullet$ is generically exact, then $M' = 0$.
\end{theorem}
\begin{proof}
Since $\varphi$ is a morphism of twisting cochains, it follows that $\nabla \varphi = 0$. Let $\cover'$ be as in Section~\ref{section:residue}. By Proposition-Definition~\ref{propdef:U-Rdef}, we have that $\nabla(U^F|_{\cover'}) = \id_F - (R^F)'$ and $\nabla(U^E|_{\cover'}) = \id_E - (R^E)'$, and hence
\begin{align*}
	\nabla (U^F \varphi U^E|_{\cover'}) &=
	(\id_F - (R^F)') \varphi U^E|_{\cover'} -
	U^F \varphi (\id_E - (R^E)')|_{\cover'} \\ &=
	\varphi U^E|_{\cover'} - U^F \varphi|_{\cover'} -
	(R^F)' \varphi U^E|_{\cover'} + U^F \varphi (R^E)'|_{\cover'}.
\end{align*}
By \eqref{eq:nabla-A-B} we get that
\[
	R(U^F \varphi U^E) =
	\varphi U^E - U^F \varphi - M' - \nabla (U^F \varphi U^E),
\]
and hence
\begin{align*}
	\nabla(M' + R(U^F \varphi U^E)) &=
	\varphi(\id_E - R^E) - (\id_F - R^F)\varphi \\ &=
	R^F \varphi - \varphi R^E.
\end{align*}
If each complex $F_\alpha^\bullet$ is generically exact, then $R' = 0$ by Proposition-Definition~\ref{propdef:U-Rdef}, and hence $M' = 0$.
\end{proof}
The results in \cite{Lar} are formulated in the setting of Remark~\ref{rmk:aw}. In particular, the complex is assumed to be generically exact. In this case Theorem~\ref{thm:comparison} coincides with \cite{Lar}*{Theorem~3.2}. As we have seen, it is natural to not require that the complex is generically surjective at level 0. In this case Theorem~\ref{thm:comparison} becomes a generalization by introducing the extra component $M'$. 

Under the conditions that the twisting cochains has singularity subvarieties, then we have the following result on which parts of the residue part of $M$ that vanish. This is the analogous result to \cite{Lar}*{Proposition~3.6}.
\begin{proposition}
\label{prop:Mvanish}
Let $(F,a)$, $(E,b)$, and $\varphi: (E,b) \to (F,a)$ be as in Theorem~\ref{thm:comparison}, and suppose that $(F,a)$ and $(E,b)$ have singularity subvarieties $Z^{F,k}$ and $Z^{E,k}$ respectively. If
\begin{align*}
	\codim Z^{E,\ell+m} &\geq m+1 \text{ for } m = 1,\dots,k-\ell-1 \text{ and} \\
	\codim Z^{F,\ell+m} &\geq m \text{ for } m = 2,\dots,k-\ell,
\end{align*}
then $R(U^F \varphi U^E)_k^\ell = 0$.
\end{proposition}
\begin{proof}
By a similar argument as in the beginning of the proof of Proposition~\ref{prop:R-vanish}, we need to show that
\begin{equation}
\label{eq:Rdelsig}
	R \left( (\dbar \sigma^F)^r \sigma^F \varphi
	(\dbar \sigma^E)^q \sigma^E \right)_k^\ell = 0
\end{equation}
for $0 \leq q+r \leq k-\ell-2$. We will do this by induction over $k$ starting with $k = \ell+2$. We have that $R(\sigma^F \varphi \sigma^E)_{\ell+2}^\ell$ has bidegree $(0,1)$ and has support on $Z^{E,\ell+1} \cup Z^{F,\ell+2}$, since $\sigma^F \varphi \sigma^E$ is smooth outside this set, which has codimension at least 2, and hence it must vanish by the dimension principle.

Assume now that $R \left( (\dbar \sigma^F)^r \sigma^F \varphi (\dbar \sigma^E)^q \sigma^E \right)_j^\ell = 0$ for $j = \ell+2,\dots,k-1$ and $0 \leq q+r \leq j-\ell-2$. When $r = 0$ in \eqref{eq:Rdelsig} we have that
\[
	R \left( \sigma^F \varphi
	(\dbar \sigma^E)^q \sigma^E \right)_k^\ell =
	\sum_{j=\ell+q+1}^{k-1}
	R \left( (\sigma^F \varphi)_k^j \left((\dbar \sigma^E)^q \sigma^E
	\right)_j^\ell \right)
\]
for $q = 0, \dots, k-\ell-2$. Outside $\bigcup_{j+1\leq m \leq k} Z^{F,m}$ we have that $(\sigma^F \varphi)_k^j$ is smooth. Thus by \eqref{eq:residueSmooth} we have that
\[
	R \left( (\sigma^F \varphi)_k^j
	\left((\dbar \sigma^E)^q \sigma^E \right)_j^\ell \right) =
	(\sigma^F \varphi)_k^j R \left((\dbar \sigma^E)^q \sigma^E \right)_j^\ell,
\]
which vanishes by Proposition~\ref{prop:R-vanish}. Thus the left-hand side, which has bidegree at most $(0,j-\ell)$, has support on $\bigcup_{j+1\leq m \leq k} Z^{F,m}$, which has codimension at least $j - \ell + 1$, and hence it must vanish by the dimension principle. Moreover, if $r \geq 1$ in \eqref{eq:Rdelsig}, then
\[
	R \left( (\dbar \sigma^F)^r \sigma^F \varphi
	(\dbar \sigma^E)^q \sigma^E \right)_k^\ell =
	\sum_{j=\ell+q+r+1}^{k-1}
	R \left( (\dbar \sigma^F)_k^j \left( (\dbar \sigma^F)^{r-1}
	\sigma^F \varphi (\dbar \sigma^E)^q \sigma^E \right)_j^\ell \right)
\]
for $0 \leq q+r \leq k-\ell-2$.
Outside $\bigcup_{j+1\leq m \leq k} Z^{F,m}$ we have that $(\dbar \sigma^F)_k^j$ is smooth. Thus by \eqref{eq:residueSmooth} we have that 
\[
	R \left( (\dbar \sigma^F)_k^j
	\left( (\dbar \sigma^F)^{r-1} \sigma^F \varphi
	(\dbar \sigma^E)^q \sigma^E \right)_j^\ell \right) =
	(\dbar \sigma^F)_k^j R
	\left( (\dbar \sigma^F)^{r-1} \sigma^F \varphi (\dbar \sigma^E)^q \sigma^E
	\right)_j^\ell,
\]
which vanishes by the induction hypothesis. Thus the left-hand side, which has bidegree at most $(0,j-\ell)$, has support on $\bigcup_{j+1\leq m \leq k} Z^{F,m}$, which has codimension at least $j-\ell+1$, and hence it must vanish by the dimension principle.
\end{proof}
\begin{corollary}
Let $(F,a)$ and $(E,b)$ be twisted resolutions of coherent $\holo_X$-modules $\sheafF$ and $\sheafE$, and let $\varphi: (E,b) \to (F,a)$ be a morphism. Then
\begin{equation}
\label{eq:Mellk}
	M_k^\ell = 0
	\text{ for }
	\ell = 1,\dots,k-2.
\end{equation}
Moreover, if $\sheafF$ and $\sheafE$ have codimension $\geq k$, then
\[
	M_k^0 = 0.
\]
\end{corollary}
\begin{proof}
As in Corollary~\ref{cor:Rvanish} we have that $\codim Z^{F,j} \geq j$ and $\codim Z^{E,j} \geq j$. By Proposition~\ref{prop:Mvanish} we get that $R(U^F \varphi U^E)_k^\ell = 0$ for $\ell = 1,\dots,k-2$. Since $(F,a)$ and $(E,b)$ are twisted resolutions, we have that $(M')_k^\ell = 0$ for $\ell \geq 1$. Thus \eqref{eq:Mellk} holds.

From \cite{Eis}*{Corollary~20.12} it also follows that $\codim Z^{F,j} \geq \codim \sheafF$ and $\codim Z^{E,j} \geq \codim \sheafE$. Applying Proposition~\ref{prop:Mvanish} we get that $M_k^0 = 0$.
\end{proof}
\begin{remark}
In this paper we have given results on the vanishing of certain components of the residue parts of the currents $R$ and $M$. For a more complete description of these currents, one needs also to analyze the parts $R'$ and $M'$. For example, when $R$ is the residue current associated with a twisted resolution of a coherent $\holo_X$-module $\sheafF$, one crucial property of $R'$ is that it takes values in $\Hom(F^0, F)$, see \cite{JL}*{Proposition~4.4}. If $(F,a)$ and $(E,b)$ in Theorem~\ref{thm:comparison} are twisted resolutions of the coherent $\holo_X$-modules $\sheafF$ and $\mathcal{E}$ respectively, from \eqref{eq:M'} it follows that $M'$ takes values in $\Hom(E^0, F)$ since $U^E$ has degree $-1$, and $R^F$ and $R^E$ take values in $\Hom(F^0, F)$ and $\Hom(E^0, E)$, respectively.

For the general case, when $R$ is the residue current associated with a twisted resolution of a complex or more generally a twisting cochain, we have no such analogous descriptions of the parts $R'$ and $M'$. These questions are therefore interesting for future work related to this topic.
\end{remark}
\section{Morphisms of twisted resolutions}
\label{section:morphisms}
The utility of Theorem~\ref{thm:comparison} relies on the existence of a morphism of twisting cochains. There is a natural situation where one can show that such morphisms exist. The following result follows directly from \cite{Wei}*{Proposition~2.33}.
\begin{proposition}
\label{prop:morphism}
Let $\sheafF^\bullet$ and $\mathcal{E}^\bullet$ be complexes of coherent $\holo_X$-modules, and let $(F,a)$ and $(E,b)$ be twisted resolutions. If $f: \mathcal{E}^\bullet \to \mathcal{F}^\bullet$ is a chain map, then there exists a morphism $\varphi: (E,b) \to (F,a)$.
\end{proposition}
\begin{proof}[Proof of Theorem~1.2.]
The proof follows by Theorem~\ref{thm:comparison} and Proposition~\ref{prop:morphism}.
\end{proof}
We find it illustrative to give a direct proof of Proposition~\ref{prop:morphism} in the following simpler case following the same ideas as in \cite{OTT}*{pp. 229--231}.
\begin{proposition}
\label{prop:morphism2}
Let $(F,a)$ and $(E,b)$ be twisted resolutions of the coherent $\holo_X$-modules $\sheafF^\bullet$ and $\mathcal{E}^\bullet$. If $f: \mathcal{E} \to \mathcal{F}$ is a morphism, then there exists a morphism $\varphi: (E,b) \to (F,a)$.
\end{proposition}
We can use the comparison formula to relate $R^F$ to $R^E$ via $f$. We get
\[
	R^F \varphi^0 - \varphi R^E = \nabla M.
\]
Note that the first term only involves $\varphi^0$ since $R^F$ is associated with a twisted resolution of a coherent $\holo_X$-module and hence only takes values in $\Hom(F^0,F)$.
For the proof the following lemma will be needed.
\begin{lemma}[\cite{OTT}*{Lemma 1.6}]
\label{lemma:hom-complex}
Let $\opens$ be a Stein manifold, and let $(F^\bullet,a^0)$ and $(E^\bullet,b^0)$ be finite complexes of holomorphic vector bundles over $\opens$, with $E^r = 0$ for $r > 0$ and such that $F^\bullet$ is exact as a complex of sheaves in negative degrees. Then the complex $\Homs^\bullet(E,F)(\opens)$ with differential $\partial$ given by
\[
	\partial \varphi \defeq
	a^0 \varphi - (-1)^{\deg{\varphi}} \varphi b^0
\]
is exact in negative degrees.
\end{lemma}
\begin{proof}[Proof of Proposition~\ref{prop:morphism2}]
It is well known that over each $\opens_\alpha$ one can find a chain map $\varphi_\alpha^0: (E_\alpha^\bullet, b_\alpha^0) \to (F_\alpha^\bullet, a_\alpha^0)$ that extends $f$ over $\opens_\alpha$. We wish to show that $\varphi^0$ can be extended to a morphism $\varphi: (E,b) \to (F,a)$.

By definition we have that such a morphism $\varphi$ satisfies $D \varphi = 0$ if and only if
\begin{equation}
\label{eq:D-closed}
	\delta \varphi^{k-1} +
	\sum_{j=0}^k a^{k-j} \varphi^j -
	\sum_{j=0}^k \varphi^j b^{k-j} = 0
\end{equation}
for all $k \geq 0$. Here and throughout the proof we take $\varphi^k$ to be zero for $k < 0$. For $k = 0$, we have that \eqref{eq:D-closed} is satisfied by the above.

The existence of $\varphi^1$ satisfying \eqref{eq:D-closed} follows from the fact that both $a_{\alpha \beta}^1 \varphi_{\beta}^0$ and $\varphi_\alpha^0 b_{\alpha \beta}^1$ are liftings of the map $f: \mathcal{E}|_{\opens_{\alpha \beta}} \to \sheafF|_{\opens_{\alpha \beta}}$, and hence they are $\partial$-homotopic, i.e., there exists $\varphi_{\alpha \beta}^1$ such that
\[
	a_{\alpha \beta}^1 \varphi_{\beta}^0 - \varphi_\alpha^0 b_{\alpha \beta}^1 =
	a_{\alpha}^0 \varphi_{\alpha \beta}^1 - \varphi_{\alpha \beta}^1 b_{\beta}^0.
\]

For $k \geq 2$, we will construct the $\varphi^k$ inductively as follows. For a given $m \geq 2$ suppose that $\varphi^0, \dots, \varphi^{m-1}$ are so that \eqref{eq:D-closed} is satisfied for $0 \leq k \leq m-1$. This is our induction hypothesis. In view of \eqref{eq:D-closed} we must now show that we can find $\varphi^m$ such that
\begin{equation}
\label{eq:partial}
	a^0 \varphi^m - \varphi^m b^0 =
	-\left(
	\delta \varphi^{m-1} +
	\sum_{j=0}^{m-1} a^{m-j} \varphi^j -
	\sum_{j=0}^{m-1} \varphi^j b^{m-j}
	\right).
\end{equation}

Note that over each $\opens_{\alpha_0 \dots \alpha_m}$ the left-hand side of \eqref{eq:partial} equals
\[
	(-1)^m a_{\alpha_0}^0 \varphi_{\alpha_0 \dots \alpha_m}^m -
	\varphi_{\alpha_0 \dots \alpha_m}^m b_{\alpha_m}^0 =
	(-1)^m \partial \varphi_{\alpha_0 \dots \alpha_m}^m,
\]
where $\partial$ denotes the differential for the complex $\Homs^\bullet(E_{\alpha_m},F_{\alpha_0})(\opens_{\alpha_0 \dots \alpha_m})$. Define
\begin{equation}
\label{eq:rho}
	\rho^m \defeq
	\delta \varphi^{m-1} +
	\sum_{j=0}^{m-1} a^{m-j} \varphi^j -
	\sum_{j=0}^{m-1} \varphi^j b^{m-j}.
\end{equation}
By Lemma~\ref{lemma:hom-complex}, it follows that $\varphi^m$ exists if $\partial \rho^m_{\alpha_0 \dots \alpha_m} = 0$ for all $(\alpha_0,\dots,\alpha_m)$.
We claim that
\[
	a^0 \rho^m + \rho^m b^0 = 0.
\]
If this claim holds, then over each $\opens_{\alpha_0 \dots \alpha_m}$, since $\rho_{\alpha_0 \dots \alpha_m}^m$ belongs to $\Homs^{1-m}(E_{\alpha_m},F_{\alpha_0})(\opens_{\alpha_0 \dots \alpha_m})$, we have that
\begin{align*}
	(-1)^m a_{\alpha_0}^0 \rho_{\alpha_0 \dots \alpha_m}^m +
	\rho_{\alpha_0 \dots \alpha_m}^m b_{\alpha_0}^0 &=
	(-1)^m \left(
	a_{\alpha_0}^0 \rho_{\alpha_0 \dots \alpha_m}^m -
	(-1)^{1-m} \rho_{\alpha_0 \dots \alpha_m}^m b_{\alpha_0}^0
	\right) \\ &=
	(-1)^m \partial \rho_{\alpha_0 \dots \alpha_m}^m = 0
\end{align*}
as required.

It remains to prove the claim. We begin by examining the expression
\begin{equation}
\label{eq:arho}
	a^0 \rho^m =
	a^0 \delta \varphi^{m-1} +
	a^0\sum_{j=0}^{m-1} a^{m-j} \varphi^j -
	a^0\sum_{j=0}^{m-1} \varphi^j b^{m-j}.
\end{equation}
For the second term in \eqref{eq:arho}, first note that
\[
	a^0 a^{m-j} =
	-\left(\delta a^{m-j-1} + \sum_{k=0}^{m-j-1} a^{m-j-k} a^k \right)
\]
by \eqref{eq:twisting-cochain2}. This gives that the second term equals
\begin{align*}
	&a^0 \sum_{j=0}^{m-1} a^{m-j} \varphi^j \\ &=
	-\sum_{j=0}^{m-1}
	\left(
	\left(
	\delta a^{m-j-1} + \sum_{k=0}^{m-j-1} a^{m-j-k} a^k
	\right) \varphi^j \right) \\ &=
	-\sum_{j=0}^{m-1}
	\left(
	\left(
	\delta a^{m-j-1} + \sum_{k=1}^{m-j-1} a^{m-j-k} a^k
	\right) \varphi^j +
	a^{m-j} a^0 \varphi^j
	\right) \\ &=
	-\sum_{j=0}^{m-1}
	\left(
	\left(
	\delta a^{m-j-1} + \sum_{k=1}^{m-j-1} a^{m-j-k} a^k
	\right) \varphi^j -
	a^{m-j}
	\left(
	\delta \varphi^{j-1} +
	\sum_{k=0}^{j-1} a^{j-k} \varphi^k -
	\sum_{k=0}^j \varphi^k b^{j-k}
	\right)
	\right),
\end{align*}
where the last equality follows by the induction hypothesis. We have that the first and second double-sum cancel. In view of \eqref{eq:deltader} and the fact that the $a^j$ and $\varphi^j$ have degree 1 and 0 resepectively we get that the above expression equals
\begin{align*}
	&-\sum_{j=0}^{m-1}
	\delta a^{m-j-1} \varphi^j +
	\sum_{j=0}^{m-2}
	a^{m-j-1} \delta \varphi^j -
	\sum_{j=0}^{m-1} \sum_{k=0}^j
	a^{m-j} \varphi^k b^{j-k} \\ = &-
	\sum_{j=0}^{m-1}
	\delta a^{m-j-1} \varphi^j +
	\sum_{j=0}^{m-1}
	a^{m-j-1} \delta \varphi^j -
	a^0 \delta \varphi^{m-1} -
	\sum_{j=1}^{m-1} \sum_{k=0}^{j-1}
	a^{m-j} \varphi^k b^{j-k} -
	\sum_{j=0}^{m-1}
	a^{m-j} \varphi^j b^0 \\ = &-
	\delta \left(
	\sum_{j=0}^{m-1}
	a^{m-j-1} \varphi^j
	\right) -
	a^0 \delta \varphi^{m-1} -
	\sum_{j=0}^{m-1}
	a^{m-j} \varphi^j b^0 - S,
\end{align*}
where
\[
	S \defeq
	\sum_{j=1}^{m-1} \sum_{k=0}^{j-1}
	a^{m-j} \varphi^k b^{j-k} =
	\sum_{j=1}^{m-1} \sum_{k=0}^{j-1}
	a^{j-k} \varphi^k b^{m-j}.
\]
Similarly, for the third term in \eqref{eq:arho}, we have that
\begin{align*}
	-a^0 \sum_{j=0}^{m-1} \varphi^j b^{m-j} &=
	\sum_{j=0}^{m-1}
	\left(
	\delta \varphi^{j-1} +
	\sum_{k=0}^{j-1} a^{j-k} \varphi^k -
	\sum_{k=0}^j \varphi^k b^{j-k}
	\right) b^{m-j} \\ &=
	\sum_{j=0}^{m-1}
	\delta \varphi^{j-1} b^{m-j} -
	\sum_{j=0}^{m-1} \sum_{k=0}^j
	\varphi^k b^{j-k} b^{m-j} + S \\ &=
	\sum_{j=0}^{m-2}
	\delta \varphi^j b^{m-j-1} -
	\sum_{j=0}^m \sum_{k=0}^j
	\varphi^k b^{j-k} b^{m-j} +
	\sum_{k=0}^m
	\varphi^k b^{m-k} b^0 + S \\ &=
	\sum_{j=0}^{m-2}
	\delta \varphi^j b^{m-j-1} -
	\sum_{k=0}^m \varphi^k
	\sum_{j=0}^{m-k}
	b^j b^{m-k-j} +
	\sum_{k=0}^{m-1}
	\varphi^k b^{m-k} b^0 + S.
\end{align*}
Since $b$ is a twisting cochain, by \eqref{eq:twisting-cochain2} we get that the above expression equals
\begin{align*}
	&\sum_{j=0}^{m-2}
	\delta \varphi^j b^{m-j-1} +
	\sum_{j=0}^m \varphi^j
	\delta b^{m-j-1} +
	\sum_{j=0}^{m-1}
	\varphi^j b^{m-j} b^0 + S \\ =
	&\sum_{j=0}^{m-1}
	\delta \varphi^j b^{m-j-1} +
	\sum_{j=0}^{m-1}
	\varphi^j \delta b^{m-j-1} -
	\delta \varphi^{m-1} b^0 +
	\sum_{j=0}^{m-1}
	\varphi^j b^{m-j} b^0 + S \\ =
	&\delta \left(
	\sum_{j=0}^{m-1}
	\varphi^j b^{m-j-1}
	\right) - \left(
	\delta \varphi^{m-1} -
	\sum_{j=0}^{m-1}
	\varphi^j b^{m-j}
	\right) b^0 + S,
\end{align*}
where we in the last equality have used that the $b^j$ and $\varphi^j$ have degree 1 and 0 respectively. Adding all three terms in \eqref{eq:arho} gives
\begin{align*}
	a^0 \rho^m &= 
	\delta \left(
	-\sum_{j=0}^{m-1}
	a^{m-j-1} \varphi^j +
	\sum_{j=0}^{m-1}
	\varphi^j b^{m-j-1}
	\right) -
	\left(
	\delta \varphi^{m-1} +
	\sum_{j=0}^{m-1} a^{m-j} \varphi^j -
	\sum_{j=0}^{m-1} \varphi^j b^{m-j}
	\right) b^0 \\ &=
	\delta \delta \varphi^{m-2} - \rho^m b^0 = - \rho^m b^0.
\end{align*}
Thus the claim follows.
\end{proof}

We end this paper with the proof of Theorem~1.3. We will need the following lemma, which follows directly from \cite{Wei}*{Proposition~2.33}.
\begin{lemma}
\label{lemma:phi-homotopy}
Let $(F,a)$ and $(E,b)$ be twisted resolutions of a complex of coherent $\holo_X$-modules $\sheafF^\bullet$. Then there exist morphisms $\varphi: (E,b) \to (F,a)$ and $\psi: (F,a) \to (E,b)$ such that $\psi \varphi$ and $\varphi \psi$ are homotopic to the identity.
\end{lemma}
\begin{proof}[Proof of Theorem~\ref{thm:Rhomotopy}]
By Lemma~\ref{lemma:phi-homotopy} we have that $\varphi \psi$ is homotopic to the identity, i.e.,
\[
	\varphi \psi - \id = D \alpha = \nabla \alpha
\]
for some holomorphic $\alpha$. The desired homotopy is obtained by multiplying $\psi$ from the right on both sides of \eqref{eq:Rhomotopy}. The statement that $R^E$ is homotopic to $\psi R^F \varphi$ follows analoguosly.
\end{proof}

\end{document}